\documentclass[11pt,oneside,reqno]{amsart}
\usepackage{amssymb, amscd, amsmath, amsthm, latexsym, hyperref,xcolor,accents, mathtools}
\usepackage{pb-diagram, fancyhdr, graphicx, psfrag}
\usepackage[text={6.3in,8.5in},centering,letterpaper,dvips]{geometry}
\usepackage[
textwidth=0.9in,
backgroundcolor=yellow,
bordercolor=orange,
textsize=scriptsize]{todonotes}
\usepackage{standalone}
\usepackage{subcaption}

\def\cfk{{\textrm{CFK}}}

\newcommand{\spinc}{\ifmmode{{\mathfrak s}}\else{${\mathfrak s}$\ }\fi}
\newcommand{\spinct}{\ifmmode{{\mathfrak t}}\else{${\mathfrak t}$\ }\fi}
\newcommand{\spincw}{\ifmmode{{\mathfrak w}}\else{${\mathfrak w}$\ }\fi}
\def\Z{\mathbb Z}

\def\cfki{\cfk^\infty(K)}
\def\U{\Upsilon}

\newtheorem{theorem}{Theorem}[section]

\newtheorem*{theorem*}{Theorem}
\newtheorem*{notation*}{Notation and Conventions}
\newtheorem{lemma}[theorem]{Lemma}

\newtheorem{proposition}[theorem]{Proposition}
\newtheorem{conjecture}[theorem]{Conjecture}
\newtheorem{corollary}[theorem]{Corollary}

\theoremstyle{definition}

\newtheorem{example}[theorem]{Example}
\theoremstyle{remark}
\newtheorem*{ack}{Acknowledgements}

\numberwithin{equation}{section}

\def\U{\Upsilon}

\begin{document}
\title[Concordances from differences of torus knots to $L$--space knots]{Concordances from differences of torus knots to $L$--space knots}
\author{Samantha Allen}
\address{Samantha Allen: Department of Mathematics, Dartmouth College, Hanover, NH 03755 }
\email{samantha.g.allen@dartmouth.edu}

\begin{abstract} 
It is known that connected sums of positive torus knots are not concordant to $L$--space knots.  Here we consider differences of torus knots.  The main result states that the subgroup of the concordance group generated by two positive torus knots contains no nontrivial $L$--space knots other than the torus knots themselves.  Generalizations to subgroups generated by more than two torus knots are also considered.
\end{abstract}

\maketitle



\section{Introduction}

An important theme in knot concordance theory has been how the three-dimensional properties of knots affect the knot concordance group.  For instance, in one direction we have the results of Kirby-Lickorish and Myers \cite{kirby-lickorish, myers} that every knot concordance class is represented by a (prime) hyperbolic knot.  In the other direction, there are the results of Litherland \cite{litherland} that positive torus knots are linearly independent in the concordance group and that the quotient of the concordance group by the subgroup generated by torus knots is infinitely generated.

Heegaard Floer theory has provided powerful new tools to investigate this idea. For instance, Friedl-Livingston-Zentner \cite{friedl-livingston-zentner} show that alternating knots generate a subgroup with infinitely generated quotient in the concordance group.  Aceto-Alfieri \cite{aceto-alfieri} have studied the question (given in \cite{friedl-livingston-zentner}) of which sums of torus knots are concordant to alternating knots.  Similar questions are addressed in, for instance, \cite{aceto-celoria-park, alfieri, alfieri-kang-stipsicz, yozgyur}.   Here, we study the question of which torus knots are concordant to $L$-space knots.

From the perspective of Heegaard Floer theory, an important family of knots are $L$--space knots \cite{os3}.  This family includes all positive torus knots. In \cite{krcatovich}, Krcatovich showed that all $L$--space knots are prime.  Thus no nontrivial connected sum of knots is an $L$--space knot.  We consider instead concordances from knots to $L$--space knots.  In \cite{Zemke} it was shown that certain connected sums of torus knots are not concordant to $L$--space knots, and this was expanded in \cite{livingston18} to show that no positive linear combination of positive torus knots is concordant to an $L$--space knot.  In this paper we extend this to the case of differences, showing that no nontrivial differences of (multiples) of pairs of torus knots can be concordant to an $L$--space knot.

\begin{theorem}
If the connected sum of distinct positive torus knots $mT(p,q)\,\#\, nT(r,s)$ is concordant to an $L$--space knot, then either $m=0$ and $n=1$ or $m=1$ and $n=0$.
\label{main theorem}
\end{theorem}

\noindent  The proof uses an approach similar to that of Livingston in \cite{livingston18}, making use of the Levine-Tristram signature function, the Ozsv\'ath--Szab\'{o} tau invariant, the Alexander polynomial, and two additional items:
\begin{itemize}
\item properties of the  Ozsv\'{a}th--Stipsicz--Szab\'{o} Upsilon invariant
\item and the result of Hedden and Watson \cite{hedden-watson} that the leading terms of the Alexander polynomial of an $L
$--space knot of genus $g$ must be $t^{2g}-t^{2g-1}$.  
\end{itemize}
More generally, we make the following conjecture:
\begin{conjecture}
If a connected sum of (possibly several) torus knots is concordant to an $L$--space knot, then it is concordant to a positive torus knot.
\label{conjecture}
\end{conjecture}

\noindent Note that since torus knots are linearly independent in the concordance group, a connected sum of torus knots is concordant to a positive torus knot $T(p,q)$ only if it is of the form  $$mT(p,q)\,\#\, (1-m) T(p,q) \,\#\, T(p_1, q_1) \,\# -T(p_1, q_1) \,\#\, \dots \,\#\, T(p_N, q_N)\,\# -T(p_N, q_N)$$ where $m\geq1$.  

As progress towards this conjecture, we give conditions under which a connected sum of  many torus knots  is not concordant to an $L$--space knot.  These conditions involve all of the aforementioned invariants, as well as the relations among Upsilon functions for torus knots discovered by Feller and Krcatovich  \cite{feller-krcatovich}.

\begin{notation*}
Throughout this paper, all torus knots $T(a, b)$ considered are such that $1<a<b$ and $\gcd(a,b) = 1$.   All Alexander polynomials are normalized to be polynomials (rather than as symmetric Laurent polynomials) with positive constant terms.    In addition, we abuse the ``big O'' notation $\mathcal{O}(t^k)$ to mean ``terms of degree greater than or equal to $k$''.
\end{notation*}

\begin{ack}
This problem was suggested by Charles Livingston and the work greatly benefited from discussions with him.  Thanks are also due to the referee for very detailed and helpful corrections and suggestions.
\end{ack}

\section{Preliminaries}

In \cite{os4}, Ozsv\'{a}th and Szab\'{o} introduced the Heegaard Floer invariant $\widehat{\mathit{HF}}(Y)$ which associates a graded abelian group to a closed $3$--manifold $Y$.  A rational homology $3$--sphere $Y$ is called an $L$--space if rank$\left(\widehat{\mathit{HF}}(Y)\right) = |H_1(Y;\mathbb{Z})|$ (see \cite{os3}).  A knot $K$ is called an $L$--space knot if it admits a positive $L$--space surgery.  Since lens spaces are $L$-spaces, positive torus knots are $L$--space knots.  In this section, we gather some useful facts and references concerning $L$--space knots and torus knots, beginning with their knot complexes. 

The Heegaard Floer knot complex $\cfki$ was introduced in \cite{os2}.  For $L$--space knots, the complex $\cfki$ is always a {\it staircase} complex (as in Figure \ref{staircase}) where the height and width of each step is determined by the gaps in the exponents of the Alexander polynomial of $K$. The Alexander polynomial of an $L$--space knot can be written as 
$$\Delta_K(t)=\sum_{i=0}^d (-1)^i t^{a_i}$$
for some sequence of integers $\{a_i\}$.  The complex $\cfki$ is a staircase of the form $$[a_1-a_0, a_2-a_1,\dots,a_d-a_{d-1}]$$ where the indices alternate between horizontal and vertical steps.  For more details, see \cite{os3} and \cite{borodzik-livingston}.

\begin{figure}[ht]
  \centering
  \resizebox{5cm}{5cm}{     \begin{tikzpicture}
        \coordinate (Origin)   at (0,0);
        \coordinate (XAxisMin) at (-1,0);
        \coordinate (XAxisMax) at (11,0);
        \coordinate (YAxisMin) at (0,-1);
        \coordinate (YAxisMax) at (0,11);
        \draw [thin, black,-latex] (XAxisMin) -- (XAxisMax);
        \draw [thin, black,-latex] (YAxisMin) -- (YAxisMax);
        \clip (-1,-1) rectangle (11,11);
       
       \foreach \x in {0,...,10}
     		\draw (\x,1pt) -- (\x,-3pt)
			node[anchor=north] {\x};
    	\foreach \y in {0,...,10}
     		\draw (1pt,\y) -- (-3pt,\y) 
     			node[anchor=east] {\y};

        \node[draw,circle,fill=white, scale=0.5] (A) at (0,10) {};
        \node[draw,circle,fill,scale=0.5] (B) at (1,10) {};
        \node[draw,circle,fill=white,scale=0.5] (C) at (1,6) {};
        \node[draw,circle,fill,scale=0.5] (D) at (3,6) {};
        \node[draw,circle,fill=white,scale=0.5] (E) at (3,3) {};
        \node[draw,circle,fill,scale=0.5] (F) at (6,3) {};
        \node[draw,circle,fill=white,scale=0.5] (G) at (6,1) {};
        \node[draw,circle,fill,scale=0.5] (H) at (10,1) {};
        \node[draw,circle,fill=white,scale=0.5] (I) at (10,0) {};
      
        \path [very thick] [->] (B) edge node[left] {} (A);
        \path [very thick] [->] (B) edge node[left] {} (C);
        \path [very thick] [->] (D) edge node[left] {} (C);
        \path [very thick] [->] (D) edge node[left] {} (E);
        \path [very thick] [->] (F) edge node[left] {} (E);
        \path [very thick] [->] (F) edge node[left] {} (G);
        \path [very thick] [->] (H) edge node[left] {} (G);
        \path [very thick] [->] (H) edge node[left] {} (I);


    \end{tikzpicture}}
  \caption{$\cfk ^{\infty}(T(5,6))$ is a staircase of the form $[1, 4, 2, 3, 3, 2, 4, 1]$.}
  \label{staircase}
\end{figure}
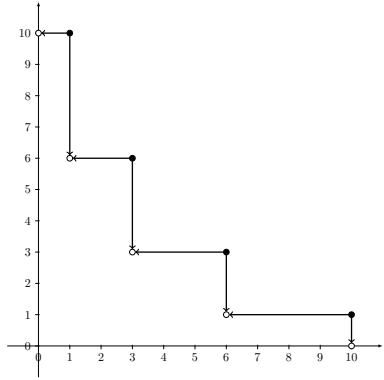

In \cite{os2}, Ozsv\'{a}th and Szab\'{o} defined the tau invariant, $\tau(K)$.  This is a concordance invariant which is additive under connected sum.  In \cite{oss}, Ozsv\'{a}th, Stipsicz, and Szab\'{o} defined the $Upsilon$ invariant, $\U_K(t)$, a piecewise-linear function with domain $[0,2]$.  The tau invariant of a knot $K$ is equal to the negative of the slope of $\U_K(t)$ near $t=0$.  See \cite{oss} for explicit computations for the family of knots $T(p, p+1)$ and for formulas for computing $\U$ for $L$--space knots from $\mathrm{CFK}^{\infty}$.  The derivative $\U_K'(t)$ is piecewise--constant with singularities at $t$--values where the slope changes in $\U_K(t)$.  See \cite{livingston17} for results on computing $\U_K'(t)$ from $\mathrm{CFK}^{\infty}(K)$.

Finally, consider the Levine--Tristram signature function $\sigma_K(t)$.  The signature function is piecewise--constant and integer--valued with possible jumps occurring at zeroes of the Alexander polynomial of $K$.  Livingston in \cite{livingston18} showed that if the cyclotomic polynomial $\phi_c(t)$ divides $\Delta_{T(p,q)}(t)$, then $\sigma_{T(p,q)}(t)$ jumps by $-2$ at $t=1/c$.  We use this fact along with the following results to prove Theorem \ref{main theorem}.

\begin{theorem}[\cite{hedden-watson, livingston17, os3}]
If $K$ is an $L$--space knot, then
\begin{itemize}
\item[(a)] $2 \tau (K) = \deg\left(\Delta_K(t)\right)$.
\item[(b)] $\U_K'(t)$ is increasing.
\item[(c)] $\Delta_K(t)$ has lowest order terms $1-t$ and highest order terms $t^{2g}-t^{2g-1}$, where $g=g(K)$ is the genus of the knot.
\end{itemize}
\label{facts}
\end{theorem}

Much of the work in this paper relies on understanding the Alexander polynomial of torus knots.  We list here three useful characterizations.  These are all classical results; see, for example, \cite{rolfsen, wall}.
\begin{theorem}  Let $T(p,q)$ be a torus knot with $0<p<q$ and $\gcd (p,q) =1$.  The following hold:
\begin{enumerate}
\item $\displaystyle{\Delta_{T(p,q)}(t) = \frac{(t^{pq}-1)(t-1)}{(t^p-1)(t^q-1)}}$.
\item $\displaystyle{\Delta_{T(p,q)}(t) = \prod _{\substack{h|p, \, \ell|q,\\ h, \ell \,\neq\, 1}}\phi_{h\ell}(t)}$ where $\phi_n(t)$ is the $n$th cyclotomic polynomial.
\item $\displaystyle{\Delta_{T(p,q)}(t) = \sum_{s\in S_{p,q}} (t^s-t^{s+1})}$ where $S_{p,q} = \{ap+bq \mid a, b \in \Z_{\geq 0}\}$.
\end{enumerate}
\end{theorem}

\section{Proof of the main theorem}
We break the proof into several smaller propositions.  Note that if $J$ is concordant to $K$, then $\tau(K) = \tau(J)$, $\sigma_K(t) = \sigma_J(t)$ (away from roots of $\Delta_K(t)$), and $\U_K(t) = \U_J(t)$, as these are all concordance invariants.

Livingston's result in \cite{livingston18} includes the cases where both $m$ and $n$ are nonnegative, so we need only show the result for at most one of $m,n$ positive.   The first case is easy; we use the Ozsv\'{a}th--Szab\'{o} tau invariant to rule out the case where $m,n\leq 0$. 

\begin{proposition}
The knot $K = mT(p,q)\,\#\, nT(r,s)$ where $m,n\leq 0$ is not concordant to a nontrivial $L$--space knot.
\label{no negs}
\end{proposition}

\begin{proof}
If $m=n=0$, then $K$ is the unknot.  So assume that $m,n\leq 0$ with at most one of $m,n$ equal to 0.  Suppose that $J$ is a nontrivial $L$--space knot concordant to $K$ and consider the Ozsv\'{a}th--Szab\'{o} tau invariant, $\tau(K)$.  Recall that $0<p<q$ and $0<r<s$.  Since the knots $J$, $T(p,q)$, and $T(r,s)$ are nontrivial $L$--space knots, their tau invariants $\tau (J)$, $\tau(T(p,q))$, and $\tau (T(r,s))$ are positive by Theorem \ref{facts} (a).  However, $J$ is concordant to $K$, so by the additivity of $\tau$ under forming connected sums, we have that
$$\tau (J) = \tau (K) =m\tau (T(p,q)) +n\tau (T(r,s)) <0,$$ 
a contradiction.  Thus $K$ is not concordant to an $L$--space knot.
\end{proof}

The remaining cases are those where $K = mT(p,q)\,\#\, nT(r,s)$ and $m\cdot n < 0$.  Without loss of generality, we will assume that $m>0$ and $n<0$.  For ease of notation, we write this as $K = mT(p,q)\,\#\, {-nT(r,s)}$ with $m,n>0$.  Next, we use the  Ozsv\'{a}th--Stipsicz--Szab\'{o} Upsilon invariant to rule out the case where $r>p$.

\begin{proposition}
The knot $K =mT(p,q)\,\#\, {{-nT(r,s)}}$, where $m,n> 0$ and $r>p$, is not concordant to an $L$--space knot.
\label{p<r}
\end{proposition}

\begin{proof}
Suppose that $J$ is a $L$--space knot concordant to $K$ and consider the Ozsv\'{a}th--Stipsicz--Szab\'{o} Upsilon invariant, $\Upsilon_K(t)$.  Recall that $0<p<q$ and $0<r<s$.  Since the knots $J$, $T(p,q)$, and $T(r,s)$ are all nontrivial $L$--space knots, $\U'_J(t)$, $\Upsilon '_{T(p,q)}(t)$, and $\Upsilon '_{T(r,s)}(t)$ must be increasing.  Analyzing $\mathrm{CFK}^{\infty}(T(p,q))$ and $\mathrm{CFK}^{\infty}(T(r,s))$, we see that $\Upsilon '_{T(p,q)}(t)$ has its first jump at $t=2/p$ and $\Upsilon '_{T(r,s)}(t)$ has its first jump at $t=2/r$.  Since $\U$ is additive under connected sum, $\U'_K(t)$ has its first jump at $t=$ min$\{2/p, 2/r\} = 2/r$.  Because $-n<0$, $\Upsilon '_{{-nT(r,s)}}(t)$ is decreasing and this first jump is negative, implying that $\U'_K(t)$ is not increasing.  Since $\U_J(t) = \U_K(t)$, we have that $\U'_J(t)$ is not increasing, which is a contradiction.
\end{proof}

For the remaining cases, we rely heavily on the work of Livingston in \cite{livingston18} for analyzing the relationship between the Alexander polynomial and the Levine--Tristram signature function of connected sums of torus knots.

\begin{proposition}
If the knot  $K =mT(p,q)\,\#\, {-nT(r,s)}$,  where $m,n\geq 1$ and $r \leq p$, is concordant to an $L$--space knot $J$, then $rs$ divides $pq$ and 
$$\frac{\left(\Delta_{T(p,q)}(t)\right)^m}{\left(\Delta_{T(r,s)}(t)\right)^{n}} = \Delta_J(t).$$
\label{alex poly}
\end{proposition}

\begin{corollary}
If the knot  $K =mT(p,q)\,\#\, {-nT(r,s)}$,  where $m,n\geq 1$ and $r\leq p$, is concordant to an $L$--space knot $J$, then $m = n+1$.
\label{m=n+1}
\end{corollary}

\begin{proof}[Proof of Proposition \ref{alex poly}]
Suppose that $J$ is an $L$--space knot concordant to $K$.  Recall that $0<p<q$ and $0<r<s$.  Consider the Alexander polynomial of $K$.  It is a product of cyclotomic polynomials $\phi _c(t)$:
$$\Delta_K(t) = \left(\Delta_{T(p,q)}(t)\right)^m \left(\Delta_{T(r,s) }(t)\right)^{n} = \left(\prod _{\substack{h|p, \, \ell|q,\\ h, \ell \,\neq\, 1}}\phi_{h\ell}(t) \right)^m\cdot \left(\prod _{\substack{h|r, \, \ell|s,\\ h, \ell \,\neq\, 1}}\phi_{h\ell}(t) \right)^{n}.$$
Let
\begin{gather*}
C =  \{ c : \left. \phi_{c}(t) \mid \Delta_{T(r,s)}(t) \right. \text{ and } \left. \phi_{c}(t) \mid \Delta_{T(p,q)}(t) \right.\}, \\
C_{p,q} =  \{ c : \left. \phi_{c}(t) \nmid \Delta_{T(r,s)}(t) \right. \text{ and } \left. \phi_{c}(t) \mid \Delta_{T(p,q)}(t) \right.\}, \\
C_{r,s} =  \{ c : \left. \phi_{c}(t) \mid \Delta_{T(r,s)}(t) \right. \text{ and } \left. \phi_{c}(t) \nmid \Delta_{T(p,q)}(t) \right.\} .
\end{gather*}
By \cite{livingston18}, for every $c\in C$, the Levine--Tristram signature function of $K$, $\sigma_K(t)$, jumps by $-2(m-n)$ at $t=1/c$.  Therefore, since $J$ is concordant to $K$ and so their signature functions are equal away from roots of their Alexander polynomials, we would have that $\left(\phi_{c}(t)\right)^{|m-n|}$ divides $\Delta_J(t)$.  Similarly, if $c\in C_{r,s}$, then $\left(\phi_{c}(t)\right)^{n}$ divides $\Delta_J(t)$, or if $c\in C_{p,q}$, then $\left(\phi_{c}(t)\right)^{m}$ divides $\Delta_J(t)$.  Thus
\begin{align}
\text{deg}(\Delta_J(t)) &\geq |m-n|\sum_{c\in C} \text{deg}(\phi_c(t))+n\sum_{c\in C_{r,s}} \text{deg}(\phi_c(t)) +m\sum_{c\in C_{p,q}} \text{deg}(\phi_c(t)) \nonumber \\
&=(m+n-2\min\{m,n\})\sum_{c\in C} \text{deg}(\phi_c(t))+n\sum_{c\in C_{r,s}} \text{deg}(\phi_c(t)) +m\sum_{c\in C_{p,q}} \text{deg}(\phi_c(t)) \nonumber  \\
&=\text{deg}(\Delta_K(t)) - 2\min\{m,n\}\sum_{c\in C} \text{deg}(\phi_c(t)).
\label{deg ineq}\end{align}
On the one hand, we know that for $L$--space knots the degree of the Alexander polynomial is equal to twice the $\tau$ invariant of the knot.  So, since $J$ is concordant to $K$, we have that
\begin{equation}\text{deg}(\Delta_J(t)) = 2\tau(J) = 2\tau(K) = 2\left(m\cdot\frac{(p-1)(q-1)}{2}-n\cdot\frac{(r-1)(s-1)}{2}\right).\label{deg J}\end{equation}
On the other hand, we have
$$\text{deg}(\Delta_K(t)) = m\left(\text{deg}\Delta_{T(p,q)}(t)\right) + n\left(\text{deg}\Delta_{T(r,s)}(t)\right) = m(p-1)(q-1)+n(r-1)(s-1).$$
Therefore, by Equation \eqref{deg ineq},
\begin{align*}\sum_{c\in C} \text{deg}(\phi_c(t)) &\geq \frac{\text{deg}(\Delta_K(t))-\text{deg}(\Delta_J(t))}{2\min\{m,n\}}=\frac{2n(r-1)(s-1)}{2\min\{m,n\}} \\ &\geq (r-1)(s-1) = \text{deg}(\Delta_{T(r,s)}(t)).\end{align*} 
However, by the definition of $C$, 
$$\sum_{c\in C} \text{deg}(\phi_c(t)) \leq \text{deg}(\Delta_{T(r,s)}(t)),$$
and so
$$\sum_{c\in C} \text{deg}(\phi_c(t)) = \text{deg}(\Delta_{T(r,s)}(t)).$$
Thus it must be that $\Delta_{T(r,s)}(t) \mid \Delta_{T(p,q)}(t)$.  Note that this implies that $rs$ divides $pq$ since $\phi_{rs}(t)$ divides $\Delta_{T(r,s)}(t)$.   Also, $\Delta_{T(r,s)}(t) \mid \Delta_{T(p,q)}(t)$ implies that  $\sigma_K(t)$ jumps by $-2m$ at $1/c$ for each $\phi_c(t)$ dividing $\Delta_{T(p,q)}(t)$ and not dividing $\Delta_{T(r,s)}(t)$ and jumps by $-2(m-n)$ at $1/c$ for each $\phi_c(t)$ dividing both.  So we have that
$$\frac{\left(\Delta_{T(p,q)}(t)\right)^m}{\left(\Delta_{T(r,s)}(t)\right)^{n}} \text{ divides } \Delta_J(t).$$
From Equation \eqref{deg J}, we see that 
$$\deg\left(\frac{\left(\Delta_{T(p,q)}(t)\right)^m}{\left(\Delta_{T(r,s)}(t)\right)^{n}}\right) = \deg\left(\Delta_J(t)\right)$$
and therefore
\begin{equation*} \frac{\left(\Delta_{T(p,q)}(t)\right)^m}{\left(\Delta_{T(r,s)}(t)\right)^{n}} = \Delta_J(t). \qedhere \end{equation*}
\end{proof}

\begin{proof}[Proof of Corollary \ref{m=n+1}]
Suppose that $J$ is an $L$--space knot concordant to $K=mT(p,q)\,\#\, {-nT(r,s)}$,  where $m,n >0$ and $r \leq p$.  Then by Proposition \ref{alex poly}, we know that
\begin{equation}
\frac{\left(\Delta_{T(p,q)}(t)\right)^m}{\left(\Delta_{T(r,s)}(t)\right)^{n}} = \Delta_J(t).
\label{alex poly eqn}
\end{equation}
Since $T(p,q)$, $T(r,s)$, and $J$ are all $L$--space knots, Theorem \ref{facts}(c) and Equation \eqref{alex poly eqn} imply that 
$$\frac{(1-t+\mathcal{O}(t^2))^m}{(1-t+\mathcal{O}(t^2))^n} = 1-t+\mathcal{O}(t^2).$$
Rearranging and expanding, we see that
$$(1-t+\mathcal{O}(t^2))^m = (1-t+\mathcal{O}(t^2))(1-t+\mathcal{O}(t^2))^n$$
$$1-mt+\mathcal{O}(t^2) = (1-t+\mathcal{O}(t^2))(1-nt+\mathcal{O}(t^2))$$
$$1-mt+\mathcal{O}(t^2) = 1-(n+1)t+\mathcal{O}(t^2).$$
So it must be that $m = n+1$.
\end{proof}

Before proving Theorem \ref{main theorem}, we show that it holds in the case where $m=2$.

\begin{lemma} [$m=2$ case] \label{m=2}
The knot $K =2T(p,q)\,\#\, {-T(r,s)}$,  where $r\leq p$ and $rs$ divides $pq$, is not concordant to an $L$--space knot unless $T(p,q) = T(r,s)$.
\end{lemma}

\begin{proof}
Suppose for contradiction that $K$ is concordant to an $L$--space knot $J$.  Consider first the case of $r<p$.  Applying Proposition \ref{alex poly}, we have that 
\begin{align*}\Delta_J(t) &= \frac{\left(\Delta_{T(p,q)}(t)\right)^2}{\Delta_{T(r,s)}(t)} = \left(\frac{(t^{pq}-1)(t-1)}{(t^p-1)(t^{q}-1)}\right)^2\frac{(t^r-1)(t^s-1)}{(t^{rs}-1)(t-1)}\\
&=\frac{(t^{pq}-1)^2(t-1)(t^r-1)(t^s-1)}{(t^{p}-1)^2(t^q-1)^2(t^{rs}-1)} = \frac{-1+t+t^r+\mathcal{O}(t^{r+1})}{-1+\mathcal{O}(t^{\,\min\{rs, p\}})}.\end{align*}
Rearranging, we get
$$(-1+\mathcal{O}(t^{\,\min\{rs, p\}}))\Delta_J(t) = -1+t+t^r+\mathcal{O}(t^{r+1}).$$
 Equating coefficients and noting that $r<\min\{rs, p\}$ in this case, the Alexander polynomial of $J$ is 
$$\Delta_J(t) = 1-t-t^r +\text{higher degree terms}.$$
Thus $J$ does not have the Alexander polynomial of an $L$--space knot, and we have reached a contradiction.

Now, consider the case of $r=p$.  Note that, since we are assuming $rs$ divides $pq$, we now have $s\mid q$ and so $r=p<s<q$.  Again, applying Proposition \ref{alex poly}, we have that 
\begin{align*}\Delta_J(t) &= \frac{\left(\Delta_{T(p,q)}(t)\right)^2}{\Delta_{T(p,s)}(t)} =\frac{(t^{pq}-1)^2(t-1)(t^s-1)}{(t^{p}-1)(t^{q}-1)^2(t^{ps}-1)}\\
&= \displaystyle{\frac{1-t-t^s+t^{s+1}+\mathcal{O}(t^{pq})}{1-t^{p}+\mathcal{O}(t^{\,\min\{q,ps\}})}}
\end{align*}
Rearranging, we get
$$(1-t^{p}+\mathcal{O}(t^{\,\min\{q,ps\}}))\Delta_J(t) =1-t-t^s+t^{s+1}+\mathcal{O}(t^{pq}).$$
Let $s=kp+i$ where $1\leq i <p$.  If $i\neq 1$, by equating coefficients we see that the Alexander polynomial of $J$ is
$$\Delta_J(t) = 1-t+t^p-t^{p+1}+t^{2p}-t^{2p+1}+\cdots +t^{kp}-t^{kp+1}-t^s +\text{higher degree terms}.$$
Thus $J$ does not have the Alexander polynomial of an $L$--space knot.
If $i = 1$,  by equating coefficients we see that the Alexander polynomial of $J$ is
$$\Delta_J(t) = 1-t+t^p-t^{p+1}+\cdots +t^{(k-1)p}-t^{(k-1)p+1}+t^{kp}-2t^s +\text{higher degree terms}.$$
Again, $J$ does not have the Alexander polynomial of an $L$--space knot, and we have reached a contradiction.
\end{proof}
Finally, we prove Theorem \ref{main theorem}.

\begin{proof}[Proof of Theorem \ref{main theorem}]
By Propositions \ref{no negs}, \ref{p<r}, and \ref{alex poly}, Corollary \ref{m=n+1}, and Lemma \ref{m=2}, we can consider only the case of  $K =mT(p,q)\,\#\, {-(m-1)T(r,s)}$,  where $m\geq 3$, $r\leq p$, and $rs$ divides $pq$.  Suppose that $J$ is an $L$--space knot concordant to $K$.  We apply Proposition \ref{alex poly} to have that 
$$\Delta_J(t) = \frac{\left(\Delta_{T(p,q)}(t)\right)^m}{\left(\Delta_{T(r,s)}(t)\right)^{m-1}} = \left(\frac{(t^{pq}-1)(t-1)}{(t^p-1)(t^q-1)}\right)^m\left(\frac{(t^r-1)(t^s-1)}{(t^{rs}-1)(t-1)}\right)^{m-1}$$
$$=\left(\frac{t^{pq}-1}{t^{rs}-1}\cdot \frac{t^{rs}-t^r-t^s+1}{t^{pq}-t^p-t^q+1}\right) ^{m-1}\left(\frac{(t^{pq}-1)(t-1)}{(t^p-1)(t^q-1)}\right)$$
$$=\frac{(t^{pq-rs}+t^{pq-2rs}+\cdots + t^{rs}+1)^{m-1}(t^{rs}-t^r-t^s+1)^{m-1}}{(t^{pq}-t^p-t^q+1)^{m-1}}\cdot \Delta_{T(p,q)}(t),$$
where the last equality is due to the fact that $rs\mid pq$.  Rearranging, expanding, and focusing on lower degree terms, we get
\begin{gather}
(t^{pq}-t^p-t^q+1)^{m-1}\Delta_J(t) = (t^{pq-rs}+t^{pq-2rs}+\cdots + t^{rs}+1)^{m-1}(t^{rs}-t^r-t^s+1)^{m-1}\Delta_{T(p,q)}(t),\nonumber \\
(1-t^p-t^q+t^{pq})^{m-1} \Delta_J(t) = (1-t^r-t^s+2t^{rs}+\mathcal{O}(t^{rs+r}) )^{m-1}\Delta_{T(p,q)}(t), \label{compare2} \\
(1-(m-1)t^p +\mathcal{O}(t^{p+1}))\Delta_J(t) = (1-(m-1)t^r +\mathcal{O}(t^{r+1}))\Delta_{T(p,q)}(t). \label{compare} 
\end{gather}
(Recall that, in our notation, $p<q$, $r<s$, gcd$(p,q) =1$, and gcd$(r,s) = 1$.)  Let $a_i$ and $b_i$ represent the coefficient of $t^i$ in $\Delta_J(t)$ and $\Delta_{T(p,q)}(t)$ respectively and note that $a_i, b_i \in \{-1, 0, 1\}$, $a_0 = b_0 = 1$, and $a_1 = b_1 = -1$.  Then the left-hand side of Equation \eqref{compare} is 
$$1-t+a_2t^2+\cdots +a_{p-1}t^{p-1}+(-(m-1)+a_p)t^p+\mathcal{O}(t^{p+1}) $$
while the right-hand side is 
$$1-t+b_2t^2+\cdots +b_{r-1}t^{r-1}+(-(m-1)+b_r)t^r+\mathcal{O}(t^{r+1}).$$
This implies that $a_i = b_i$ for $1\leq i \leq r-1$ since we have assumed $r\leq p$.  If $r<p$, then $-(m-1)+b_r = a_r$.  Recalling that 
$$\Delta_{T(p,q)}(t) = \sum_{d\in S_{p,q}} (t^d-t^{d+1})$$ where $S_{p,q} = \{xp+yq \mid x, y \in \Z_{\geq 0}\}$,  note that $b_r = 0$ when $r<p<q$.  Therefore, in this case, $-(m-1)= a_r$ and so $m\leq 2$.  Thus we are in the trivial case or the case of Lemma \ref{m=2}.

If $r=p$, then $\gcd (p, s) =1$ and $p<s<q$ (otherwise we are in the trivial case since we have assumed that $rs$ divides $pq$).  Thus $s\neq xp+yq$ for any choices of integers $x,y>0$ and so the polynomial $(1-t^p-t^q+t^{pq})^{m-1}$ has no $t^s$ term.  Let $s=hp+\ell$ with $h\geq 1$ and $0<\ell<p$.  Using the multinomial theorem to expand in each side of Equation \eqref{compare2}, we get
\begin{equation} \begin{split}
\left(1-\sum_{k=1}^h\binom{m-1}{k}t^{kp} +\mathcal{O}(t^{s+1})\right)&\Delta_J(t)\\ &= \left(1-\sum_{k=1}^h\binom{m-1}{k}t^{kp}-(m-1)t^s +\mathcal{O}(t^{s+1})\right)\Delta_{T(p,q)}(t). \end{split}\label{compare3}
\end{equation}
Again, let $a_i$ and $b_i$ represent the coefficient of $t^i$ in $\Delta_J(t)$ and $\Delta_{T(p,q)}(t)$ respectively.  We claim that for $0\leq i \leq s-1$, we have that $a_i = b_i$.   To show this, let $i^*\geq 0$ be the smallest value such that $a_{i^*}\neq b_{i^*}$. Note that for $0\leq i \leq s-1$, the coefficient of $t^{i}$ on each side of Equation \eqref{compare3}  is
$$ a_{i} - \sum_{k=1}^{\lfloor{i/p}\rfloor} \binom{m-1}{k}a_{i-kp} = b_{i} -\sum_{k=1}^{\lfloor{i/p}\rfloor} \binom{m-1}{k}b_{i-kp}.$$
By minimality of $i^*$,  $a_{i^*-kp}=b_{i^*-kp}$ for $k\geq1$.  This implies that $i^*\geq s$ since, otherwise, the equation above implies  $a_{i^*}= b_{i^*}$.

Now, we compare the coefficients of $t^s$ in Equation \eqref{compare3}:
\begin{gather*}
a_s -\sum_{k=1}^h \binom{m-1}{k}a_{s-kp} = b_s -\sum_{k=1}^h \binom{m-1}{k}b_{s-kp} - (m-1),
\end{gather*}
which implies that $a_s=b_s-(m-1)$ since $a_i = b_i$ for all $i<s$.
Again, because $\Delta_{T(p,q)}(t) = \sum_{d\in S_{p,q}} (t^d-t^{d+1})$ and $s\not\in S_{p,q}$, we have that $b_s = 0$ or $-1$ .  Thus, $m = b_s-a_s+1 \leq 2$  and we are in the trivial case or the case of Lemma \ref{m=2}.
\end{proof}


\section{Towards the more general case}
In this section, we state and prove some results which restrict concordances from more general connected sums of torus knots to $L$--space knots.  First, we give generalizations of Proposition \ref{alex poly} and Corollary \ref{m=n+1}.

\begin{proposition}
If the knot $$K = T(p_1, q_1)\,\#\, T(p_2, q_2)\,\#\,\cdots\,\#\, T(p_m, q_m)\,\#\, {-T(p_1', q_1')}\,\#\, {-T(p_2', q_2')}\,\#\,\cdots \,\#\, {-T(p_n', q_n')},$$  where $m,n\geq 1$, is concordant to an $L$--space knot $J$, then 
$$\frac{\prod_{i=1}^m\Delta_{T(p_i,q_i)}(t)}{\prod_{i=1}^n\Delta_{T(p_i',q_i')}(t)}= \Delta_J(t).$$
\label{gen alex poly}
\end{proposition}

\begin{corollary}
If the knot   $$K = T(p_1, q_1)\,\#\, T(p_2, q_2)\,\#\,\cdots\,\#\, T(p_m, q_m)\,\#\, {-T(p_1', q_1')}\,\#\, {-T(p_2', q_2')}\,\#\,\cdots \,\#\, {-T(p_n', q_n')},$$  where $m,n\geq 1$, is concordant to a nontrivial $L$--space knot, then $m = n+1$.
\label{gen m=n+1}
\end{corollary}

\begin{proof}[Proof of Proposition \ref{gen alex poly}]
Suppose that $J$ is an $L$--space knot concordant to $K$.  Consider the Alexander polynomial of $K$.  It is a product of cyclotomic polynomials $\phi _c(t)$:
$$\Delta_K(t) = \prod_{i=1}^m \Delta_{T(p_i,q_i)}(t)\cdot \prod_{i=1}^n \Delta_{T(p_i',q_i')}(t) =\left(\prod_{i=1}^m\,\prod _{\substack{h|p, \, \ell|q,\\ h,\ell\,\neq \,1}}\phi_{h\ell}(t)\right) \cdot \left(\prod_{i=1}^n \,\prod _{\substack{h|p_i', \, \ell|q_i',\\ h,\ell\,\neq\, 1}}\phi_{h\ell}(t)\right).$$
Let $K^+ = T(p_1, q_1)\,\#\, T(p_2, q_2)\,\#\,\cdots\,\#\, T(p_m, q_m)$ and $K^- = T(p_1', q_1')\,\#\, T(p_2', q_2')\,\#\,\cdots \,\#\, T(p_n', q_n')$.  Note that $K = K^+\,\#\, {-K^-}$.
Let 
$$k^{\pm}(c) = \max \{k\mid \left(\phi_c(t)\right)^k \text{ divides } \Delta_{K^\pm}(t)\}$$ 
and note that $$\sum_{c>1} k^\pm(c)\deg (\phi_c(t)) = \deg\left(\Delta_{K^\pm}(t)\right).$$

By \cite{livingston18}, the Levine--Tristram signature function of $K$, $\sigma_K(t)$, jumps by $-2(k^+(c)-k^-(c))$ at $t=1/c$.  Since $J$ is concordance to $K$, we  then have that $(\phi_c(t))^{\left|k^+(c)-k^-(c)\right|}$ divides $\Delta_J(t)$.  Thus,
\begin{align*}
\text{deg}(\Delta_J(t)) &\geq \sum_{c>1} \left|k^+(c)-k^-(c)\right| \deg (\phi_c(t))\\
&= \sum_{c>1}\left(k^+(c)+k^-(c)-2\min\{k^+(c), k^-(c)\}\right)\deg (\phi_c(t))\\
&= \sum_{c>1}\left(k^+(c)+k^-(c)\right)\deg (\phi_c(t))-2\sum_{c>1}\min\{k^+(c), k^-(c)\}\deg (\phi_c(t))\\
&= \deg\left(\Delta_K(t)\right)-2\sum_{c>1}\min\{k^+(c), k^-(c)\}\deg (\phi_c(t)).
\end{align*}

On the one hand, we know that for $L$--space knots the degree of the Alexander polynomial is equal to twice the $\tau$ invariant of the knot.  So, since $J$ is concordant to $K$, we have that
\begin{equation} \label{degJ}\text{deg}(\Delta_J(t)) = 2\tau(J) = 2\tau(K) = 2\left(\sum_{i=1}^m\frac{(p_i-1)(q_i-1)}{2}-\sum_{i=1}^n\frac{(p_i'-1)(q_i'-1)}{2}\right).\end{equation}
On the other hand, we have
\begin{equation} \label{degK}\text{deg}(\Delta_K(t)) = \sum_{i=1}^m(p_i-1)(q_i-1)+\sum_{i=1}^n(p_i'-1)(q_i'-1).\end{equation}
Therefore,
\begin{align*} 
\sum_{c>1}\min\{k^+(c), k^-(c)\}\deg (\phi_c(t)) &\geq \frac{1}{2}\left(\text{deg}(\Delta_K(t))-\text{deg}(\Delta_J(t))\right)\\
 &= \sum_{i=1}^n(p_i'-1)(q_i'-1)\\
 &= \text{deg}(\Delta_{K^-}(t)).
\end{align*}
Notice also that 
$$\sum_{c>1}\min\{k^+(c), k^-(c)\}\deg (\phi_c(t)) \leq \sum_{c>1} k^-(c)\deg (\phi_c(t)) = \deg\left(\Delta_{K^-}(t)\right).$$
Thus, 
$$\sum_{c>1} \min\{k^+(c), k^-(c)\}\deg (\phi_c(t)) = \deg\left(\Delta_{K^-}(t)\right)$$
which implies that $k^+(c)\geq k^-(c)$ for all $c>1$.  Thus $\Delta_{K^-}(t)$ divides $\Delta_{K^+}(t).$
In addition, $k^+(c)\geq k^-(c)$ implies that $$ \prod_{c>1}(\phi_c(t))^{\left|k^+(c)-k^-(c)\right|} = \prod_{c>1}(\phi_c(t))^{\left(k^+(c)-k^-(c)\right)} =  \frac{\Delta_{K^+}(t)}{\Delta_{K^-}(t)},$$ 
and we can conclude that 
$$\frac{\Delta_{K^+}(t)}{\Delta_{K^-}(t)} \text{ divides } \Delta_J(t).$$
Comparing the degrees of the polynomials from Equations \eqref{degJ} and \eqref{degK}, we find that
$$\frac{\Delta_{K^+}(t)}{\Delta_{K^-}(t)} = \Delta_J(t),$$
as asserted.
\end{proof}

\begin{proof}[Proof of Corollary \ref{gen m=n+1}]
Suppose that $J$ is an $L$--space knot concordant to $K$.  Then by Proposition \ref{gen alex poly} we know that
\begin{equation}
\frac{\prod_{i=1}^m\Delta_{T(p_i,q_i)}(t)}{\prod_{i=1}^n\Delta_{T(p_i',q_i')}(t)}= \Delta_J(t).
\label{gen alex poly eqn}
\end{equation}
Since $T(p_i,q_i)$, $T(p_i',q_i')$, and $J$ are all $L$--space knots, Theorem \ref{facts}(c) and Equation \eqref{gen alex poly eqn} imply that 
$$\frac{\prod_{i=1}^m(1-t+\mathcal{O}(t^2))}{\prod_{i=1}^n(1-t+\mathcal{O}(t^2))}= 1-t+\mathcal{O}(t^2).$$
Rearranging and expanding, we see that
$$\prod_{i=1}^m(1-t+\mathcal{O}(t^2)) = (1-t+\mathcal{O}(t^2))\prod_{i=1}^n(1-t+\mathcal{O}(t^2))$$
$$1-mt+\mathcal{O}(t^2) = (1-t+\mathcal{O}(t^2))(1-nt+\mathcal{O}(t^2))$$
$$1-mt+\mathcal{O}(t^2) = 1-(n+1)t+\mathcal{O}(t^2).$$
So it must be that $m = n+1$.
\end{proof}

Lastly, we use the following result of Feller and Krcatovich to give a condition on the Upsilon function of 
$$K = T(p_1, q_1)\,\#\, T(p_2, q_2)\,\#\,\cdots\,\#\, T(p_m, q_m)\,\#\, {-T(p_1', q_1')}\,\#\, {-T(p_2', q_2')}\,\#\,\cdots \,\#\, {-T(p_n', q_n')},$$
under which $K$ cannot be an $L$--space knot.
\begin{theorem}[Feller--Krcatovich, \cite{feller-krcatovich}]\label{feller-krcatovich}
Let $p<q$ be coprime integers.  Then 
$$\U_{T(p, \, q)}(t)=\U_{T(p,\, q-p)}(t)+\U_{T(p, \, p+1)}(t).$$
\end{theorem}

\begin{theorem} Suppose that
$$K = T(p_1, q_1)\,\#\, T(p_2, q_2)\,\#\,\cdots\,\#\, T(p_m, q_m)\,\#\, {-T(p_1', q_1')}\,\#\, {-T(p_2', q_2')}\,\#\,\cdots \,\#\, {-T(p_n', q_n')}$$
with 
$$\U _K (t) = c_1\U _{T(a_1, a_1+1)}(t) + c_2\U _{T(a_2, a_2+1)}(t)+\cdots + c_s\U _{T(a_r, a_r+1)}(t)-c_1'\U _{T(a_1', a_1'+1)}(t) $$
$$- c_2'\U _{T(a_2', a_2'+1)}(t)-\cdots - c_s'\U _{T(a_s', a_s'+1)}(t)$$
where $c_i, c_i' >0$ for all $i$, $a_1>a_2>\cdots >a_r$, and $a_1'>a_2'>\cdots >a_s'$.  If there exists $a_i'$ such that $a_i'$ does not divide $a_j$ for any $j$, then $K$ is not concordant to an $L$--space knot.
\label{upsilon}
\end{theorem}

\begin{proof}[Proof of Theorem \ref{upsilon}]
For an $L$--space knot $J$, $\U_J'(t)$ is an increasing, piecewise--constant function.  The Upsilon function $T(p, p+1)$ then has increasing, piecewise--constant derivative.  By analyzing $\mathrm{CFK}^{\infty}(T(p,p+1))$, we also know that $\U_{T(p,p+1)}'(t)$ has jumps in $[0,1]$ only at $t = 2i/p$ for $i$ such that $0<2i/p\leq 1$. Let $K^+ = T(p_1, q_1)\,\#\, T(p_2, q_2)\,\#\,\cdots\,\#\, T(p_m, q_m)$ and $K^- = T(p_1', q_1')\,\#\, T(p_2', q_2')\,\#\,\cdots \,\#\, T(p_n', q_n').$

 If $\U_K(t)$ is as stated in the theorem, with $a_i'$ such that $a_i'$ does not divide $a_j$ for any $j$, we have that $\U_{K^-}'(t)$ has a positive jump at $2/a_i'$.  Since no $a_j$ is divisible by $a_i'$, we have that $2k/a_j \neq 2/a_i'$ for any $j$.  Thus $\U_{K^+}'(t)$ is constant at $2/a_i'$.  This implies that $\U'_K(t) =\U'_{K^+}(t) -\U'_{K^-}(t)$ has a negative jump at $2/a_i'$ and so $\U_K(t)$ is not the Upsilon function of an $L$--space knot.
\end{proof}

\begin{figure}[ht]
    \centering
    \captionsetup[subfigure]{labelformat=empty}
    \begin{subfigure}{0.45\textwidth}
        \centering
        \includegraphics[width=0.9\textwidth]{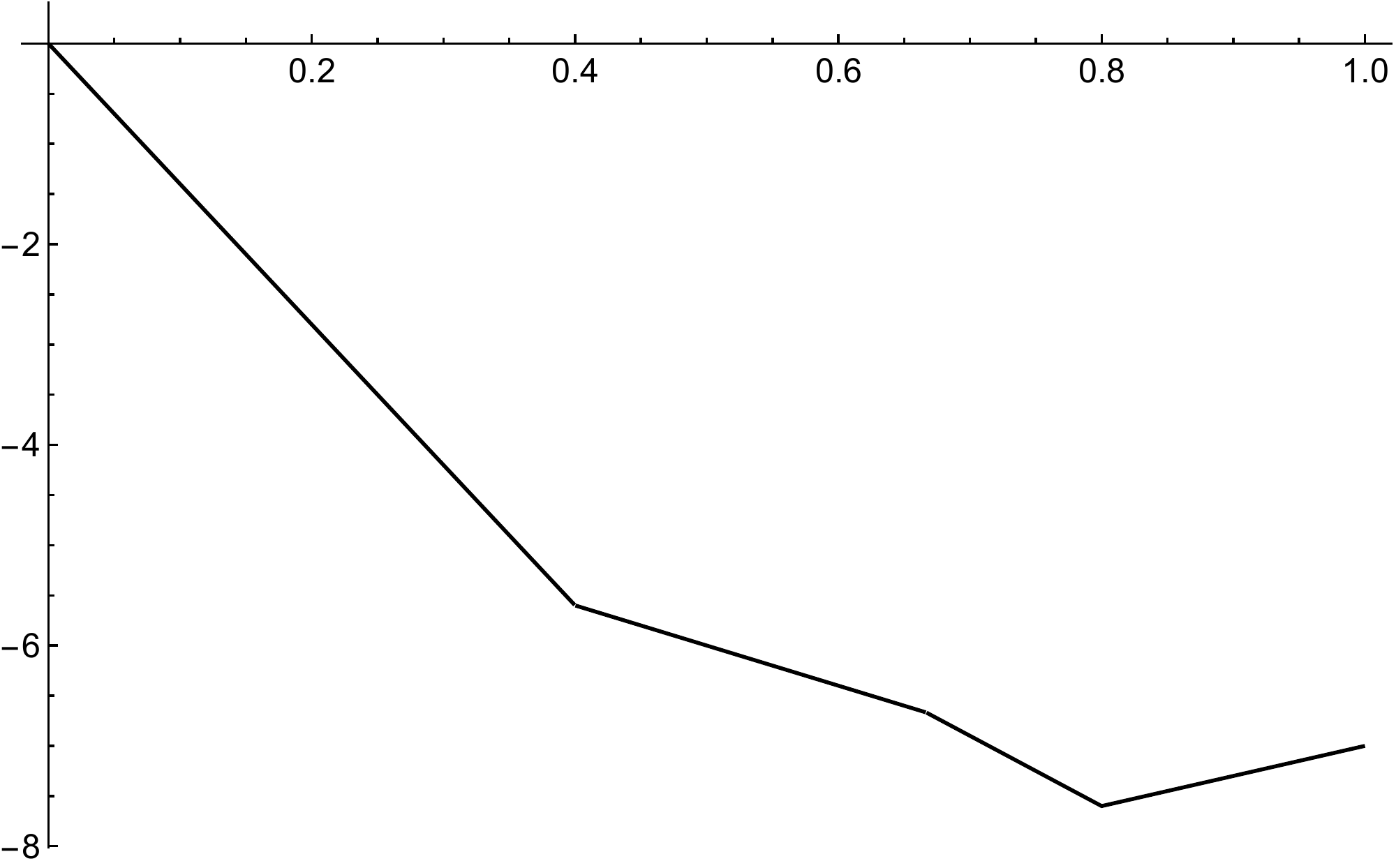} 
        \caption{$\U_K(t)$}
    \end{subfigure}
    \begin{subfigure}{0.45\textwidth}
        \centering
        \includegraphics[width=0.9\textwidth]{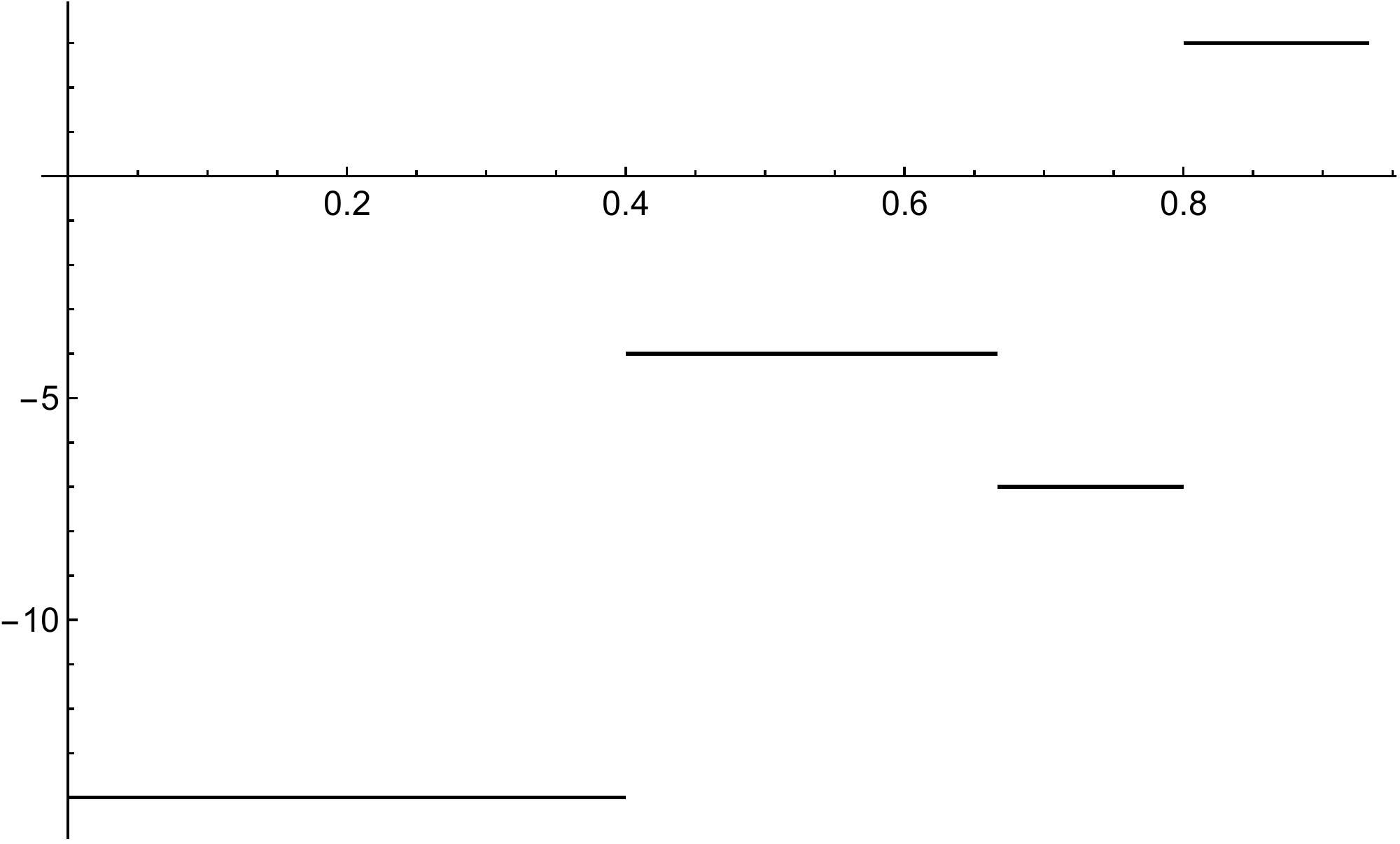}
        \caption[]{$\U'_K(t)$}
    \end{subfigure}
    \caption{The Upsilon function and its derivative for the knot $K = 3T(5,6)\,\#\, {-T(2,5)}\,\#\, {-T(3,5)}$ on the interval $[0,1]$.}
    \label{ups-sum}
\end{figure}

\begin{example}
Consider the knot $K = 3T(5,6)\,\#\, {-T(2,5)}\,\#\, {-T(3,5)}$.  Applying Theorem \ref{feller-krcatovich}, we have that 
$$\U_K(t) = 3\U_{T(5,6)}(t)-\U_{T(3,4)}(t) -3\U_{T(2,3)}(t).$$
Note that $3$ does not divide $5$.  In Figure \ref{ups-sum}, we see that $\U'_K(t)$ is not an increasing function---it has a negative jump at $t =2/3$ which is a jump in the derivative of the Upsilon function of $T(3,4)$ but not in that of $T(5,6)$.
\end{example}

\bibliography{mybiblio}
\bibliographystyle{abbrv} 

\end{document}